\documentclass[12pt]{amsart}
\usepackage[colorlinks=true, pdfstartview=FitV, linkcolor=blue, citecolor=blue, urlcolor=blue]{hyperref}

\usepackage{amssymb,amsmath, amscd}
\usepackage{times, verbatim}
\usepackage{graphicx}
\usepackage[english]{babel}
 \usepackage[usenames, dvipsnames]{color}
\usepackage{amsmath,amssymb,amsfonts}
\usepackage{enumerate}
\usepackage{anysize}

\usepackage{multirow}

\marginsize{3cm}{3cm}{3cm}{3cm}
\input xy
\xyoption{all}
\usepackage{pb-diagram}
\usepackage[all]{xy}
\input xy
\xyoption{all}

\DeclareFontFamily{OT1}{rsfs}{}
\DeclareFontShape{OT1}{rsfs}{n}{it}{<-> rsfs10}{}
\DeclareMathAlphabet{\mathscr}{OT1}{rsfs}{n}{it}

\begin{document}
\theoremstyle{plain}

\newtheorem{theorem}{Theorem}
\newtheorem{thm}{Theorem}
\newtheorem{prop}{Proposition}
\newtheorem{cor}{Corollary}
\newtheorem{conj}{Conjecture}
\newtheorem{lemma}{Lemma}
\newtheorem{definition}{Definition}
\theoremstyle{definition}
\newtheorem{question}{Question}

\theoremstyle{definition}
\newtheorem{remark}{Remark}
\theoremstyle{definition}
\newtheorem{example}{Example}

\newcommand{\Hecke}{\mathcal{H}}
\newcommand{\Liea}{\mathfrak{a}}
\newcommand{\Cmg}{C_{\mathrm{mg}}}
\newcommand{\Cinftyumg}{C^{\infty}_{\mathrm{umg}}}
\newcommand{\Cfd}{C_{\mathrm{fd}}}
\newcommand{\Cinftyfd}{C^{\infty}_{\mathrm{ufd}}}
\newcommand{\sspace}{\Gamma \backslash G}
\newcommand{\PP}{\mathcal{P}}
\newcommand{\bfP}{\mathbf{P}}
\newcommand{\bfQ}{\mathbf{Q}}
\newcommand{\Siegel}{\mathfrak{S}}
\newcommand{\g}{\mathfrak{g}}
\newcommand{\A}{{\rm A}}
\newcommand{\B}{{\rm B}}
\newcommand{\Q}{\mathbb{Q}}
\newcommand{\Gm}{\mathbb{G}_m}
\newcommand{\kk}{\mathfrak{k}}
\newcommand{\nn}{\mathfrak{n}}
\newcommand{\tF}{\tilde{F}}
\newcommand{\p}{\mathfrak{p}}
\newcommand{\m}{\mathfrak{m}}
\newcommand{\bb}{\mathfrak{b}}
\newcommand{\Ad}{{\rm Ad}\,}
\newcommand{\ttt}{\mathfrak{t}}
\newcommand{\frakt}{\mathfrak{t}}
\newcommand{\U}{\mathcal{U}}
\newcommand{\Z}{\mathbb{Z}}
\newcommand{\bfG}{\mathbf{G}}
\newcommand{\bfT}{\mathbf{T}}
\newcommand{\R}{\mathbb{R}}
\newcommand{\ST}{\mathbb{S}}
\newcommand{\h}{\mathfrak{h}}
\newcommand{\bC}{\mathbb{C}}
\newcommand{\C}{\mathbb{C}}
\newcommand{\E}{\mathbb{E}}
\newcommand{\F}{\mathbb{F}}
\newcommand{\N}{\mathbb{N}}
\newcommand{\qH}{\mathbb {H}}
\newcommand{\temp}{{\rm temp}}
\newcommand{\Hom}{{\rm Hom}}
\newcommand{\ndeg}{{\rm ndeg}}
\newcommand{\Aut}{{\rm Aut}}
\newcommand{\Ext}{{\rm Ext}}
\newcommand{\Nm}{{\rm Nm}}
\newcommand{\End}{{\rm End}\,}
\newcommand{\Ind}{{\rm Ind}\,}
\def\circG{{\,^\circ G}}
\def\M{{\rm M}}
\def\diag{{\rm diag}}
\def\Ad{{\rm Ad}}
\def\G{{\rm G}}
\def\H{{\rm H}}
\def\SL{{\rm SL}}
\def\PSL{{\rm PSL}}
\def\GSp{{\rm GSp}}
\def\PGSp{{\rm PGSp}}
\def\Sp{{\rm Sp}}
\def\St{{\rm St}}
\def\GU{{\rm GU}}
\def\ind{{\rm ind}}
\def\SU{{\rm SU}}

\newcommand{\Sym}{\mathbb{S}}
\def\U{{\rm U}}
\def\Spin{{\rm Spin}}
\def\GL{{\rm GL}}
\def\PGL{{\rm PGL}}
\def\GSO{{\rm GSO}}
\def\Gal{{\rm Gal}}
\def\SO{{\rm SO}}
\def\O{{\rm  O}}
\def\sym{{\rm sym}}
\def\St{{\rm St}}
\def\tr{{\rm tr\,}}
\def\ad{{\rm ad\, }}
\def\Ad{{\rm Ad\, }}
\def\rank{{\rm rank\,}}

\subjclass{Primary 11F70; Secondary 22E55}

\title[Multiplicities for tensor products on Special linear versus Classical groups]
      { Multiplicities for tensor products on Special linear versus Classical groups}

      \begin{abstract}
        In this paper, using computations done
        through the LiE software, we compare tensor product of irreducible selfdual representations of the special linear group with those of classical groups to formulate some conjectures relating the two. In the process a few other phenomenon present themselves which we record as questions.

        More precisely, under the natural correspondence of irreducible finite dimensional
        selfdual representations of $\SL_{2n}(\C)$ with those of $\Spin_{2n+1}(\C)$, it is easy to see
        that if the tensor product of three irreducible representations of $\Spin_{2n+1}(\C)$ contains the trivial representation,
        then so does the tensor product of the corresponding representations of $\SL_{2n}(\C)$. The paper formulates
        a conjecture in the reverse direction. We also deal with the pair $(\SL_{2n+1}(\C), \Sp_{2n}(\C))$. 
      \end{abstract}
\author{Dipendra Prasad and Vinay Wagh}
  \address{Indian Institute of Technology Bombay, Powai, Mumbai-400076} 
\address{Tata Institute of Fundamental
  Research, Colaba, Mumbai-400005.}

  \address{Indian Institute of Technology, Guwahati, Guwahati-781039} 
  \email{prasad.dipendra@gmail.com}
  \email{vinay\_wagh@yahoo.com}

  \maketitle

{\hfill \today}
\setcounter{tocdepth}{1}
\tableofcontents

\section{Introduction}
There is by now a well-known theory relating irreducible, finite dimensional representations of a group such as $\SL_{2n}(\C)$
which are selfdual, equivalently,
invariant under the outer automorphism $g \rightarrow \theta(g) = {}^t\! g^{-1}$, with irreducible, finite dimensional
representations of $\Spin_{2n+1}(\C)$, call this correspondence $\pi^{\SL} \leftrightarrow \pi^{\Spin}$, which has the following character
relationship relating character of the representation $\pi$ of $\SL_{2n}(\C)$ extended to
$\SL_{2n}(\C) \rtimes \langle \theta \rangle$ on the disconnected component with character theory of $\Spin_{2n+1}(\C)$:
$$\Theta(\pi^{\SL})(t \cdot \theta) = \Theta(\pi^{\Spin})(t');$$
here the map $t \rightarrow t'$ is a surjective
homomorphism from, say the diagonal torus $T $ in $\SL_{2n}(\C)$ to the corresponding
diagonal torus $T_\sigma$ in 
$\Spin_{2n+1}(\C)$
with kernel the subgroup of $T$ given by $t/\sigma(t)$, where $\sigma$ is the involution on   $T \subset \SL_{2n}(\C)$ given by
$(t_1,t_2,\cdots, t_{2n}) \rightarrow (t^{-1}_{2n},\cdots, t^{-1}_2,t^{-1}_1)$. We refer to the paper \cite{KLP} for the general context of a group
together with a diagram automorphism where such character relationships can be proved.

The paper was conceived to understand the decomposition of the tensor products of two selfdual representations
$\pi_1^{\SL}$  and  $\pi_2^{\SL}$ of $\SL_{2n}(\C)$  versus corresponding decomposition of the representations $\pi_1^{\Spin}$ and
$\pi_2^{\Spin}$ of    $\Spin_{2n+1}(\C)$. It is easy to see --- and has been observed by others, see for example
\cite{HS}, as well as Theorem 5.6 in \cite{H} --- that for
any representation $$\pi_3^{\Spin} \subset \pi_1^{\Spin} \otimes \pi_2^{\Spin},$$ of  $\Spin_{2n+1}(\C)$, appearing with multiplicity
$m(\pi_3^{\Spin}, \pi_1^{\Spin} \otimes \pi_2^{\Spin}) \not = 0,$
we have:
$$\pi_3^{\SL} \subset \pi_1^{\SL} \otimes \pi_2^{\SL},$$ 
for the representations of  $\SL_{2n}(\C)$,  appearing with multiplicity
$m(\pi_3^{\SL}, \pi_1^{\SL} \otimes \pi_2^{\SL}) \not = 0$.

However, what came as quite a bit of surprise, after much computer assisted checks, that the converse of the above assertion,
i.e., $$m(\pi_3^{\Spin}, \pi_1^{\Spin} \otimes \pi_2^{\Spin}) \not = 0  \iff 
m(\pi_3^{\SL}, \pi_1^{\SL} \otimes \pi_2^{\SL}) \not = 0,$$
holds, or rather, almost holds,
 which egged our curiosity on, and in the process we hope we have found something of some interest.

We end the introduction by mentioning that irreducible components of the tensor product of two irreducible representations of a simple group 
are reasonably well understood through the `Saturation conjecture', a theorem for $\SL_n(\C)$ due to Knutson and Tao \cite{KT}, and due to several
works of P. Belkale, J. Hong, S. Kumar and  L. Shen, among others,
for classical groups, see e.g. \cite{BK}, \cite{HS}, \cite{H}, \cite{Ku2}, \cite{Ku-St}, and
 \cite{Ku} for a general survey. However, the more precise question on what the multiplicities look like has not been attempted as much. Looking at
 tables of tensor product multiplicities, one is dazzled by the wealth of data 
on  statistical behaviour of multiplicities it contains. No precise question has been
attempted as far as we know.
In our work relating tensor product of $\SL_{2n}(\C)$ and $\Spin_{2n+1}(\C)$
(and more generally a pair of groups $G,G_\sigma$ that we will soon come to),
we are concerned (though not this very precise question) with how often are multiplicities in the tensor product for
$\SL_{2n}(\C)$ are even versus odd: a well-known question in the theory of partition functions which is very much present in representation theory
of $\SL_{2n}(\C)$ through Kostant partition function.

\section {Relating multiplicities} \label{bc}

Let $ G$ be a connected reductive algebraic group over $\C$.  Assume that
$\sigma$ is a diagram automorphism of $G$ of order 2.  Thus we assume that
$(B,T, S)$ is a fixed Borel subgroup $B$ in $G$, containing a maximal torus $T$,
and with $S$ a fixed pinning on $B$, which is an identification of each simple
root of $T$ on the Lie algebra of $B$ with the additive group $\C$, and that
$(B,T, S)$ is left invariant under $\sigma$.

Suppose that $\pi$ is an irreducible representation of $G$ which is invariant
under $\sigma$, and thus extends to a natural representation of
$\tilde{G}=G \rtimes \langle \sigma\rangle$ by demanding that the action of
$\sigma$ on the highest $B$-weight of $\pi$ is trivial.  Denote the
representation of $\tilde{G}$ so obtained as $\tilde{\pi}$.  The representation
$\tilde{\pi}$ of $\tilde{G}$ gives rise to an irreducible representation of the group $G_\sigma$ constructed in
\cite{KLP} that we will call $\pi'$, such that:
\[ \Theta(\pi)(t \cdot \sigma) = \Theta(\pi')(t'), \tag{0} \] where $t \in T$ and
$t'$ is the image of $T$ under the natural surjective map from $T$ to $T_\sigma$,  which is a maximal torus in
$ G_\sigma$, and is the maximal quotient of $T$ on which $\sigma$ operates trivially. 

The following proposition is due to \cite{HS}, see also Theorem 5.6 in \cite{H}, and Remark 1.3 in the introduction of \cite{H}.

\begin{prop} \label{mult} Suppose that $\pi_1,\pi_2,\cdots, \pi_n$
  are irreducible representations of $G$ which are all invariant under $\sigma$, giving rise to irreducible representations 
$\tilde{\pi}_1,\tilde{\pi}_2,\cdots, \tilde{\pi}_n$ of $\tilde{G}$, as well as 
representations
$\pi'_1,\pi'_2,\cdots, \pi'_n$ of $G_\sigma$. Let $m  (\pi_1,\pi_2,\cdots, \pi_n)$
be the multiplicity of the trivial
representation of $G$ in $\pi_1 \otimes \pi_2 \otimes \cdots \otimes \pi_n$,
$m  (\tilde{\pi}_1, \tilde{\pi}_2,\cdots, \tilde{\pi}_n)$
be the multiplicity of the trivial
representation of $\tilde{G}  $ in $\tilde{\pi}_1 \otimes \tilde{\pi}_2 \otimes \cdots \otimes \tilde{\pi}_n$,
and 
$m  (\pi'_1,\pi'_2,\cdots, \pi'_n)$
be the multiplicity of the trivial
representation of $G_\sigma$ in $\pi'_1 \otimes \pi'_2 \otimes \cdots \otimes \pi'_n$. Then,
$$2 m  (\tilde{\pi}_1, \tilde{\pi}_2,\cdots, \tilde{\pi}_n) 
= m  (\pi'_1,\pi'_2,\cdots, \pi'_n)
+ m(\pi_1,\pi_2,\cdots, \pi_n).$$
  \end{prop}
\begin{proof}
  Decompose the tensor product of the representations: \[\tilde{\pi}_1 \otimes \tilde{\pi}_2 \otimes \cdots \otimes \tilde{\pi}_n = \sum m(\pi) \pi, \tag{1}\]
  where $\pi$ runs over all irreducible representations of $G$. Since each of the representations $\pi_i$ is invariant
  under $\sigma$, so is their tensor product, thus we can write the equation (1) as:
  \[\tilde{\pi}_1 \otimes \tilde{\pi}_2 \otimes \cdots \otimes \tilde{\pi}_n
  = \sum_{\tilde{\pi}|_G {\rm ~ irreducible} }
  m(\tilde{\pi}) \tilde{\pi} +\sum_{\tilde{\pi}|_G {\rm ~not ~ irreducible} } m(\tilde{\pi}) \tilde{\pi}
  . \tag{2}\]
    In equation (2), both sides are representations of $\tilde{G}=G \rtimes \langle \sigma\rangle$.
  In the first sum on the
  right side of the equality,  $\sum m(\tilde{\pi}) \tilde{\pi}$ can be replaced by $\sum M(\pi) \tilde{\pi}$ where $M(\pi)$ is the $\pi$-isotypique piece
  in the tensor product of dimension $m(\pi)$ which carries the action of $\sigma$. (If 
  $\tilde{\pi}|_G$ is  irreducible, then there are two distinct irreducible representations of $\tilde{G}$:  $\tilde{\pi}$
 and $ \tilde{\pi}\otimes \epsilon$ where $\epsilon$ is the sign character of $\tilde{G}$, which have the same restriction to $G$.) 
  Calculating the character of the representations appearing on the left and on the right hand side of equation (2) at elements of the
  form $t \cdot \sigma \in T\rtimes \langle \sigma\rangle \subset   G \rtimes \langle \sigma\rangle$, using equation (0) with
  $t'=t/\sigma(t) \in T_\sigma$, we find that:
  \[ \Theta(\pi'_1 \otimes \pi_2' \otimes \cdots \pi_n')(t') = \sum_{\pi'} {\rm tr}[\sigma \circlearrowleft M(\pi) ]
  \Theta(\pi')(t'), \tag{3}\]
where ${\rm tr}[\sigma \circlearrowleft M(\pi) ]$ denotes trace of the action of $\sigma$ on $M(\pi)$.
  If in the tensor product $\tilde{\pi}_1 \otimes \tilde{\pi}_2 \otimes \cdots \otimes \tilde{\pi}_n$ of $\tilde{G}=G \rtimes \langle \sigma\rangle$,
  we have $m$ copies of the trivial representation of $\tilde{G}$, and $n$ copies of the sign character of $\tilde{G}$, it follows
  from equation (3) and linear independence of irreducible characters of $G_\sigma$ that:
  $$m  (\pi'_1,\pi'_2,\cdots, \pi'_n) = m-n = 2m-(m+n), $$
  with
  \begin{eqnarray*} m & = & m(\tilde{\pi}_1, \tilde{\pi}_2,\cdots, \tilde{\pi}_n) \\
    (m+n) & =  & m  (\pi_1,\pi_2,\cdots, \pi_n), \end{eqnarray*}
  proving the  proposition.
\end{proof}

\begin{cor} \label{mult-cor} With the notation as above,
  $$m-n = m({\pi'_1}, {\pi'_2},\cdots, \pi_n')
  \leq m({\pi_1}, {\pi_2}, \cdots, \pi_n) = m+n,$$
  and $$ m({\pi'_1}, {\pi'_2},\cdots, \pi_n') \equiv m({\pi_1}, {\pi_2}, \cdots, \pi_n)
  \bmod 2.$$
  In particular, if  $m({\pi_1}, {\pi_2}, \cdots, \pi_n) \leq 1$
  then $m({\pi'_1}, {\pi'_2},\cdots, \pi_n') \leq 1$, and
$$ m({\pi'_1}, {\pi'_2},\cdots, \pi_n') = m({\pi_1}, {\pi_2}, \cdots, \pi_n).$$
\end{cor}

\begin{remark}
  The Proposition \ref{mult} above has essentially  the same formulation as Proposition 2.1 in \cite{Pr} although
  Proposition 2.1 in \cite{Pr} appears to be more general in that it applies to a general pair $(G,H)$ of algebraic groups
  over finite fields, whereas in  our Proposition \ref{mult}, we have shied away from formulating it for general pairs, such as
  $(\GL_{n+1},\GL_n)$ since no diagram automorphism preserves such pairs, and have contented ourselves to have a formulation only for $(G\times \cdots \times G, \Delta G)$.
\end{remark}

\section{Review of the group $G_\sigma$}

In this section we follow \cite{KLP} to discuss the relationship between representations of the groups $G$ and $G_\sigma$ which requires first a review of the group $G_\sigma$.

Fix a Borel subgroup $B$ of $G$ and a maximal torus $T$ contained in $B$.
 Let $X$ be the character group
of $T$, $Y$ the co-character group of $T$. Let $R\subset X$ be the set of roots of $T$ in $G$, $R_+$ the set of positive roots defined by $B$ and $\Pi= \{\alpha_i: i \in I\}$  the set of simple roots in $R_+$. For each
$i\in I$, let $\check{\alpha}_i \in Y$ be the corresponding co-root. Thus
$(Y,X, \check{\alpha}_i, \alpha_i,  i \in I)$ is the root datum for $G$.  
  For $i \in I,$ let $x_i: \C \rightarrow G$, $y_i: \C \rightarrow G$ be a pinning on $G$ corresponding to root spaces $\alpha_i, -\alpha_i$.
  
  Now we assume that the automorphism $\sigma$ of $G$ preserves $B,T$, and there is a permutation $i \rightarrow
  \sigma(i)$ on the index set $I$ such that $\sigma(x_i(a)) = x_{\sigma(i)}(a),$ and   $\sigma(y_i(a)) = y_{\sigma(i)}(a)$ for $a \in \C$.

  Set, $$Y_\sigma = Y/(1-\sigma)Y, \,\,\, {}^\sigma X= \{\lambda \in X| \sigma(\lambda)=\lambda \}.$$

  We note that the natural perfect pairing $\langle -,- \rangle : Y \times X \rightarrow \Z$ induces a perfect
  pairing $\langle -,- \rangle: Y_\sigma \times {}^\sigma X \rightarrow \Z$; more precisely, if
  $j: Y \rightarrow Y_\sigma$, and $\iota: {}^\sigma X \rightarrow X$, then,
  \[ \langle j(a), b\rangle = \langle a, \iota(b) \rangle, \tag{1}\]
  for all $a \in Y, b \in {}^\sigma X.$

  Let $I_\sigma$ be the set of
  $\sigma$-orbits on $I$. For any $\eta \in I_\sigma$, let $\check{\alpha}_\eta$ be the image of $\check{\alpha}_i$ under $Y \rightarrow Y_\sigma$ where $i$ is any element of $\eta$. For any $\eta \in I_\sigma$, let
  $\alpha_\eta = 2^h\sum_{i \in \eta}  \alpha_i\in {}^ \sigma X$ where $h$ is the number of unordered pairs
  $i,j \in \eta $ such that $\alpha_i+\alpha_j \in R$. We have $h =0$ unless $G$ is of type $A_{2n}$, when $h=0$ for all but one, and $h=1$ for one simple root.

  The following proposition is proved in \cite{KLP}.

  \begin{prop}
    $(Y_\sigma, {}^\sigma X, \check{\alpha}_\eta, \alpha_\eta, \eta \in I_\sigma)$ is a root datum, defining the group $G_\sigma$ which is simply connected if $G$ is.
    \end{prop} 

  All the above preliminaries recalled from \cite{KLP} were done for the purpose of the following proposition which is of basic importance when we make the explicit computations using LiE software.
  
  \begin{prop} \label{Gsigma}
    For $i \in I$, let $\varpi_i$ be the fundamental dominant weight for $G$ (defined by the property
    $\langle \check{\alpha}_j, \varpi_i \rangle = \delta_{i,j}$). Similarly, for $\eta \in I_\sigma$,
    let $\varpi_\eta$ be the fundamental dominant weight for $G_\sigma$ (defined by the property
    $\langle \check{\alpha}_\eta, \varpi_{\eta'} \rangle =
    \delta_{\eta,\eta'}$). Then, under the inclusion
    of ${}^\sigma X \subset X$, we have:
    $$\varpi_\eta = \sum_{i \in \eta} \varpi_i.$$
  \end{prop}

  \begin{proof}
    It suffices to prove that
\[\langle \check{\alpha}_i, \varpi_{\eta} \rangle =  \begin{cases}
  1, & \text{if $ i \in \eta$} \\
  0, & \text{ otherwise},
\end{cases} \]
which is just equation (1) above together with the fact that
    $\langle \check{\alpha}_i, \varpi_{\eta} \rangle \geq 0$, and belongs to $\Z$.
    \end{proof}
\section{An example of twisted character}
A consequence of the Proposition \ref{Gsigma} for the pair
$(G, G_\sigma)=(\SL_{2n}(\C), \Spin_{2n+1}(\C))$,
is that the fundamental representation
$\Lambda^i(\C^{2n+1}), i \leq n-1$
of $\Spin_{2n+1}(\C))$ goes to the selfdual representation of $\SL_{2n}(\C)$
of highest weight $\varpi_i + \varpi_{2n+1-i}$ where $\varpi_i$ is the highest weight of the irreducible representation 
$\Lambda^i(\C^{2n})$ of $\SL_{2n}(\C)$, and the spin representation (a fundamental representation of $\Spin_{2n+1}(\C)$ of dimension $2^n$) goes to the representation $\Lambda^n(\C^{2n})$ of $\SL_{2n}(\C)$. We discuss this last case as an example  relating an irreducible
representation of a group $G$ invariant under
a diagram automorphism $\sigma$ with   an irreducible representation of $G_\sigma$ through the twisted character identity.

Thus our example will be for the group $G=\SL_{2n}(\C)$ realized on its
natural representation on a vector space $V$ of dimension $2n$, $\sigma$ will be the involution,
$$g\rightarrow \sigma(g)=J {\,}^t\!g^{-1} J^{-1},$$
where $J$ is the $(2n) \times (2n)$  anti-diagonal matrix (with respect to a basis of $V$ that we will denote as  $\{e_1,\cdots, e_n, e_{n+1},\cdots, e_{2n} \}$):
$$\left  ( \begin{array}{cccccccc}
{}& {} & {} & {}  &  {}   & {}  &{} & 1  \\
{}&   & {} & {}   &  {}  &  {}  & -1 &   \\
{}& {}  &   & {}  &   {} & 1   & {}   &\\
{}& {}  &   &   &   {} & {}   &{} &  \\
{}& {}  &   & {}     \cdot&  & {}   &{}   &\\
{} &   & {\cdot}     &  &  &  &    & \\
{}& 1  &   & {}  &   {} &  &       &    \\
-1 & {}   &  {}  & {}  &  & & &   
\end{array} \right),
$$

We will take the representation of $\SL_{2n}(\C)$ to be
$\Lambda^n(V) = \Lambda^n(\C^{2n})$
which we know is selfdual, thus invariant under $\sigma$. We will prove the following.

\begin{prop}
  The twisted character of the representation $\Pi=\Lambda^n(V)
  =\Lambda^n(\C^{2n})$ of $\SL_{2n}(\C)$
  is the character of the spin representation of $\Spin_{2n+1}(\C)$ of dimension
  $2^n$.
\end{prop}
\begin{proof}
  Let $\psi: \Pi \rightarrow \Pi$ be the unique linear map such that,
 \[\psi(gv) = \sigma(g) \psi(v), \tag{1}\]
 such that $\psi$ fixes the highest weight vector $v_0= e_1\wedge e_2 \wedge \cdots\wedge e_n \in \Pi$. (We will work with the group of diagonal matrices in $\SL_{2n}(\C)$ as the maximal torus, and the group
  of upper triangular matrices as the Borel subgroup of $\SL_{2n}(\C)$ both of which are left invariant under $\sigma$.) It is easily seen
  that $\psi^2=1$, and we will content ourselves with calculating the trace of $\psi$ (which we expect to find to be $2^n$!), the twisted trace is exactly the same calculation.

  Since $\sigma$ leave $T$-invariant, it operates on characters of $T$, denoted as $\chi\rightarrow \chi^\sigma$. For $i \leq n$, let $\chi_i$ denote the character of $T$ sending $(t_1,\cdots, t_n, \cdots, t_{2n})$
  to $t_i$. Similarly, let $\eta_i$ denote the character of $T$ sending $(t_1,\cdots, t_n, \cdots, t_{2n})$
  to $t_{2n+1-i}$. For $I \subset E= \{1,2,\cdots, n\}$,
  let $\chi_I$ be the character $\chi_I = \prod_I \chi_i$.
  Similarly, define the character $\eta_J$ for $J \subset \{n+1,n+2,\cdots, 2n\}$.

  Any character of $T$ appearing in $\Pi$ is of the form $\chi_I \cdot \eta_J$ for uniquely determined
  $I,J$ with $|I| +|J|=n$.

Let $w_0 = (1,2n)(2,2n-1)\cdots (n,n+1)$ be the element in the Weyl group $\Sym_{2n}$ represented by conjugation of the element $J$ on the diagonal torus $T$.

  By the equation (1), a $\chi$-eigenspace for the torus $T$ goes to the
  $(\chi^{-1})^{w_0}$-eigenspace under $\psi$. Writing $\chi = \chi_I \cdot \eta_J$, we find that
  $\psi$ preserves $\chi = \chi_I \cdot \eta_J$, eigenspace of $T$ if and only if:
\[\chi_I \cdot \eta_J = \chi^{-1}_J \cdot \eta_I^{-1},\]
 equivalently,
 \[\chi_I \cdot \eta_J = \chi_{E-J} \cdot \eta_{E-I},\]
 where we have used the fact that $\chi_E\cdot \eta_E$ is the trivial character on $T$.
 It follows  $\chi = \chi_I \cdot \eta_J$ eigenspace of $T$ is preserved under the action of $\psi$
 if and only if $I= E-J$.

 Clearly, $\chi_I \cdot \eta_{E-I}$ is obtained from the character $\chi_E$ be applying the element
 of $(\Z/2)^n \subset \Sym_{2n}$ which permutes $e_i$ and $e_{2n+1-i}$ for all elements of $I \subset E =
 \{1,\cdots, n\}$, and acts trivially on all other basis elements of $V$.

 Next observe that if for an element $w \in \Sym_{2n}$, $\sigma(w)=w$, as is the case for these elements
 coming from the subgroup $(\Z/2)^n \subset \Sym_{2n}$, from equation (1), it follows that:
 \[\psi(w\cdot v_0) =  w \cdot \psi(v_0) = w\cdot v_0, \]
 where we used the fact that $\psi(v_0) = v_0$.

 To conclude, either $\psi$ permutes two different eigenspaces of $T$, or if it fixes an eigenspace (parametrized by $(\Z/2)^n$) of $T$, it acts by 1 on it.  This completes the proof of the assertion that the character of $\psi$ is $2^n$. The general twisted character too is calculated to be:
 \[ \sum_{I \subset E} \chi_I \cdot \eta_{E-I} .\]
    \end{proof}

\section{The pair $(\SL_{2n}(\C), \Spin_{2n+1}(\C))$}

In this section  we will focus attention on the pair 
$(G,H)=(\SL_{2n}(\C),
\Spin_{2n+1}(\C))$ which are related to each other through fixed points of a diagram automorphism
on the bigger group, to be more precise, the $L$-group of $H$  has an embedding into the $L$-group of $G$,
with image which is the fixed points of a diagram automorphism on the $L$-group of $G$, see \cite{KLP} for the precise statements.

We will follow the notations of the {\sf LiE} software \cite{lie} for denoting representations of 
$\SL_{2n}(\C)$ and $ \Spin_{2n+1}(\C))$.

Let $\varpi_i$ denote the highest weights of the fundamental representations realized on the $i$-th exterior powers of the standard representation $V=\C^{2n}$, thus:
\begin{eqnarray*}
\varpi_1 & = &  e_1 \\
  \varpi_2 & = &  e_1 +e_2\\
  \varpi_3 & = &  e_1 +e_2+e_3\\
  & \cdot & \\
  \varpi_i & = &  e_1+e_2+\cdots+ e_i \\
  & \cdot & \\
  \varpi_{2n-1} & = &  e_1+e_2+\cdots +e_{2n-1}.
\end{eqnarray*}

{\sf LiE} software denotes an irreducible representation
of $\SL_{2n}(\C)$ by $[a_1,a_2,\cdots, a_{2n-1}]$ with $a_i$ integers $\geq 0$, which stands for the irreducible representation
of $\SL_{2n}(\C)$ with highest weight
$$a_1\varpi_1 + a_2\varpi_2+ \cdots + a_{2n-1}\varpi_{2n-1}.$$

Since in our work we are looking only at selfdual representations, we will have:

 $$[a_1,a_2,\cdots, a_{2n-1}] = [a_{2n-1},a_{2n-2},\cdots, a_{1}], $$
thus,
 $$[a_1,a_2,\cdots, a_{2n-1}] = [a_{1},a_{2},\cdots,a_{n-1},  a_n, a_{n-1}, \cdots, a_2, a_{1}]. $$

Next we come to the spin group $\Spin_{2n+1}(\C)$. Let $\varpi_i$, $1 \leq i \leq (n-1)$ denote the highest weights of the fundamental representations realized on the $i$-th exterior powers of the standard representation $V=\C^{2n+1}$ of
$\Spin_{2n+1}(\C)$, and let $\varpi_n$ be the highest weight of the spin representation of
$\Spin_{2n+1}(\C)$ (of dimension $2^n$), thus:
\begin{eqnarray*}
\varpi_1 & = &  e_1 \\
  \varpi_2 & = &  e_1 +e_2\\
  \varpi_3 & = &  e_1 +e_2+e_3\\
  & \cdot & \\
  \varpi_i & = &  e_1+e_2+\cdots + e_i \\
  & \cdot & \\
\varpi_{n-1} & = &  e_1+e_2+\cdots + e_{n-1}\\
  \varpi_n &=& (e_1+e_2+\cdots +e_n)/2.
\end{eqnarray*}

{\sf LiE} software denotes an irreducible representation of $\Spin_{2n+1}(\C)$ by
$[a_1,a_2,\cdots, a_{n}]$ with $a_i$ integers $\geq 0$, which stands for the
irreducible representation of $\Spin_{2n+1}(\C)$ with highest weight
$$a_1\varpi_1 + a_2\varpi_2+ \cdots + a_{n}\varpi_{n}.$$

We note the following  lemma which implies that the transfer of representations between $\SL_{2n}(\C)$ and
$\Spin_{2n+1}(\C)$ preserves central characters (which are the finite order characters by which the center
of these groups operate on the respective representations). The straightforward proof of the lemma will be omitted.

\begin{lemma} \label{central1}
  A selfdual representation  $[a_{1},a_{2},\cdots,a_{n-1},  a_n, a_{n-1}, \cdots, a_2, a_{1}] $
  of $\SL_{2n}(\C)$ has central character
  of order $\leq 2$, and has trivial central character if and only if $a_n$ is even.

  The representation  $[a_{1},a_{2},\cdots,a_{n-1},  a_n] $ of $\Spin_{2n+1}(\C)$ has central character
  of order $\leq 2$, and has trivial central character if and only if $a_n$ is even.
  \end{lemma}

\begin{conj}

  Recall first that by Proposition \ref{mult}, if for three irreducible representations of $\Spin_{2n+1}(\C)$ given by:
\begin{eqnarray*}
  \pi_1 & = & [a_1,a_2,\cdots, a_{n}], \\
  \pi_2 & = & [b_1,b_2,\cdots, b_{n}], \\
  \pi_3 & = & [c_1,c_2,\cdots, c_{n}],
\end{eqnarray*}
with corresponding
selfdual representations of $\SL_{2n}(\C)$,
\begin{eqnarray*}
  V_1 & = & [a_1,a_2,\cdots, a_{n-1}, a_{n}, a_{n-1}, \cdots, a_2,a_1] \\
  V_2 & = & [b_1,b_2,\cdots, b_{n-1}, b_{n},b_{n-1}, \cdots, b_2,b_1], \\
  V_3 & = & [c_1,c_2,\cdots, c_{n-1},c_{n},c_{n-1}, \cdots, c_2,c_1]
\end{eqnarray*}
$$ 1_{\Spin_{2n+1}(\C)} \subset \pi_1 \otimes \pi_2 \otimes \pi_3 \implies 1_{\SL_{2n}(\C)} \subset V_1 \otimes V_2 \otimes V_3.$$

We propose that if at least two of the three $a_n,b_n,c_n$ are nonzero, then the converse holds:

$$ 1_{\SL_{2n}(\C)} \subset V_1 \otimes V_2 \otimes V_3
\iff  1_{\Spin_{2n+1}(\C)} \subset \pi_1 \otimes \pi_2 \otimes \pi_3.$$ 

\end{conj}

Our next conjecture asserts that even with the condition $a_n=b_n=0$,
     for a density one set of  irreducible selfdual representations $(V_1,V_2,V_3)$
     of $\SL_{2n}(\C)$,

     $$ 1_{\SL_{2n}(\C)} \subset V_1 \otimes V_2 \otimes V_3
\iff  1_{\Spin_{2n+1}(\C)} \subset \pi_1 \otimes \pi_2 \otimes \pi_3.$$ 

To make the assertion precise,
  define the height $h(V)$ of a selfdual representation
  $V =  [a_1,a_2,\cdots, a_{n}, \cdots, a_2,a_1]$ of $\SL_{2n}(\C)$
  to be the maximum of the integers $a_i$. 

  Let
  \begin{eqnarray*}
  \pi_1 & = & [a_1,a_2,\cdots, a_{n} ], \\
  \pi_2 & = & [b_1,b_2,\cdots, b_{n}], \\
  \pi_3 & = & [c_1,c_2,\cdots, c_{n}],
\end{eqnarray*}
be three irreducible representations of $\Spin_{2n+1}(\C)$ with $a_n=b_n=0$ with corresponding
selfdual representations of $\SL_{2n}(\C)$,
\begin{eqnarray*}
  V_1 & = & [a_1,a_2,\cdots, a_{n-1}, a_{n}, a_{n-1}, \cdots, a_2,a_1] \\
  V_2 & = & [b_1,b_2,\cdots, b_{n-1}, b_{n},b_{n-1}, \cdots, b_2,b_1], \\
  V_3 & = & [c_1,c_2,\cdots, c_{n-1},c_{n},c_{n-1}, \cdots, c_2,c_1]
\end{eqnarray*}

Call a triple of such irreducible selfdual
representations $(V_1,V_2,V_3)$ of $\SL_{2n}(\C)$ (with $a_n,b_n=0$) to be {\it missing}
if $1_{\SL_{2n}(\C)} \subset V_1 \otimes V_2 \otimes V_3$,
but
$ 1_{\Spin_{2n+1}(\C)} \not \subset \pi_1 \otimes \pi_2 \otimes \pi_3$.

\begin{conj}
  Considering only irreducible selfdual representations
  $V_1,V_2,V_3$ of $\SL_{2n}(\C)$ as above (with $a_n=b_n=0$), 

\[ \lim_{X \rightarrow \infty}
\frac{ |{\rm~ missing ~triples~} (V_1,V_2,V_3) {\rm ~of ~height~}  \leq X|}
     { |(V_1,V_2,V_3) {\rm ~of ~height~}  \leq X,
       1_{\SL_{2n}(\C)} \subset V_1 \otimes V_2 \otimes V_3|} = 0.\]

     On the other hand we expect that there are some sporadic choices $(V_1,V_2)$ for which as many as half --- and no more --- of $V_3$ for which 
$ 1_{\SL_{2n}(\C)} \subset V_1 \otimes V_2 \otimes V_3$, 
     go missing, made precise in the following question:

     \[ \limsup_{V_1,V_2 \rightarrow \infty}
     \frac{ | V_3 {\rm ~such ~that~}
       (V_1,V_2,V_3) {\rm~ is ~ a ~missing ~triple~}|}
     { |V_3 {\rm~such ~that~} 1_{\SL_{2n}(\C)} \subset V_1 \otimes V_2 \otimes V_3|} = 1/2?\]
     Here $\limsup$ is taken over the set of irreducible selfdual representations
     $V_1,V_2$ of $\SL_{2n}(\C)$ as their heights tend to infinity.      (That there is about half of them missing seems to happen for the tensor product
         $[k, 0, \cdots, 0, k] \otimes [\ell, 0, \cdots, 0, \ell]$ although we have not proved this.)

     Presumably, the proof of saturation conjecture due to Knutson and Tau, cf. \cite{KT}  allows one to calculate the asymptotic behavior (as a polynomial in $X$: its degree and the leading coefficient) of 
the number of      irreducible (not necessarily selfdual)
representations $V_1,V_2,V_3$ of $\SL_{2n}(\C)$ of height $\leq X$
with $1_{\SL_{2n}(\C)} \subset V_1 \otimes V_2 \otimes V_3$. Next one would  want to do similar asymptotic
     estimates for number of selfdual representations $V_1,V_2,V_3$ of $\SL_{2n}(\C)$ with $1_{\SL_{2n}(\C)} \subset V_1 \otimes V_2 \otimes V_3$. Our data are quite insufficient to predict these asymptotic behaviours (degree of the polynomial in $X$, and the leading term) for which there might well be an existing theorem, though we have not seen one.

     \end{conj}

\vspace{5mm}

\begin{remark} \label{odd}
  A special case of the Conjecture 1 asserts that if $  \pi_1  =  [a_1,a_2,\cdots, a_{n}]$ with $a_n$ {\it odd}, then
  $$ 1_{\SL_{2n}(\C)} \subset V_1 \otimes V_2 \otimes V_3 \iff 1_{\Spin_{2n+1}(\C)} \subset \pi_1 \otimes \pi_2 \otimes \pi_3,$$
  because in this case as $1_{\SL_{2n}(\C)} \subset V_1 \otimes V_2 \otimes V_3$, both $b_n,c_n$ cannot be zero for central character
  reasons (calculated in Lemma \ref{central1}).  
  \end{remark}

\begin{remark}
  With notation as above, for a representation $\pi = [a_1, \cdots, a_n]$ of
  $\Spin_{2n+1}(\C)$, let $\pi(c)$ denote the representation of
  $\Spin_{2n+1}(\C)$ given by $[ca_1,ca_2, \cdots, ca_n]$. Then a weaker form of our Conjecture 1 is true by \cite{HS}, Proposition 1.4 (c): 
  $$ 1_{\SL_{2n}(\C)} \subset V_1 \otimes V_2 \otimes V_3 \iff 1_{\Spin_{2n+1}(\C)} \subset \pi_1(c) \otimes \pi_2(c) \otimes \pi_3(c),$$
  for $c = c_\sigma$, the `saturation factor', known to be 2 in this case. This also follows by combining \cite{BK}, Theorem 5, and \cite{Ku-St} Theorem 4.2.
  Thus our conjecture is a strengthening of these known results. In this regard,
  a conjecture of Kapovich in \cite{Kap}, the Saturation conjecture, part 3, at the very end of his paper,  seems related to Conjecture 1 above, although we do not see any direct relationship between his work and ours since his work is for any reductive group whereas we are dealing with
  reductive groups with a diagram automorphism. Also, his conditions are related to `regular highest weights', i.e., one which
  are not fixed by any nontrivial element of the Weyl group which in our notation for a representation $[a_{1},\cdots,a_{n-1},  a_n, a_{n-1}, \cdots, a_{1}] $ of $\SL_{2n}(\C)$ will mean
  that each of the coefficients $a_i$ are nonzero, whereas our condition is only that $a_n$ is nonzero.

\end{remark}

\begin{remark}
  In this paper, we have tried to understand the relationship of multiplicities for selfdual representations of $\SL_{2n}(\C)$ with those
  of $\Spin_{2n+1}(\C)$. Notice however that $\Sp_{2n}(\C) \subset \SL_{2n}(\C)$ arises
  as the fixed points of a diagram involution defined using the symplectic structure on the underlying vector space $V = \C^{2n}$ which allows for natural choices (up to conjugation by $\Sp_{2n}(\C)$)
  for a Borel subgroup and a maximal torus in the two groups given by taking a basis $\{e_1,e_2,\cdots, e_n, e_{n+1}, e_{n+2},\cdots, e_{2n}\}$ for $V$ with $\langle e_i,e_{2n+1-i}\rangle =1$, and all other
  products zero; the Borel subgroups in both  $\Sp_{2n}(\C)$ and  $\SL_{2n}(\C)$ is defined as the stabilizer of the flag:
  $$\{e_1\} \subset \{e_1,e_2\} \subset \cdots \subset  \{e_1, \cdots, e_n\}  \subset \{e_1, \cdots, e_{n+1}\} \subset \cdots \subset \{e_1, \cdots, e_{2n}\},$$ and the tori in the two groups $\Sp_{2n}(\C)$ and  $\SL_{2n}(\C)$  defined as the stabilizer of the corresponding lines $\{e_i\}$.

  Any highest weight for $\SL_{2n}(\C)$ gives rise to a highest weight for $\Sp_{2n}(\C)$, thus there is a natural map from the set of irreducible finite dimensional representations of $\SL_{2n}(\C)$
  to the set of irreducible finite dimensional representations of $\Sp_{2n}(\C)$, denote this map of representations as $\pi_\lambda \rightarrow V_\lambda$. Unlike our situation,  the map $\pi_\lambda \rightarrow V_\lambda$ from
  the set of irreducible representations of $\SL_{2n}(\C)$
  to the set of irreducible representations of $\Sp_{2n}(\C)$,
    is defined on {\it all} irreducible representations of $\SL_{2n}(\C)$.

  Theorem 23 of \cite{BK} asserts that if  $\pi_{\lambda_1} \otimes \pi_{\lambda_2} \otimes \cdots \otimes \pi_{\lambda_m}$
  has a $\SL_{2n}(\C)$ invariant vector, then so does the representation
  $V_{\lambda_1} \otimes V_{\lambda_2} \otimes \cdots \otimes V_{\lambda_m}$
 of $\Sp_{2n}(\C)$, seems related, but very different from ours! We refer to \cite{Ku} for a survey of many other instances of similar relationships on multiplicities of tensor products on different groups, but different from the one considered in this work.

 If $\pi_\lambda$ is a selfdual representation of $\SL_{2n}(\C)$, given by highest weight
 $\lambda_1 \geq \lambda_2 \geq \cdots \geq  \lambda_n \geq -\lambda_n \geq \cdots \geq -\lambda_1 $ where we assume that
\[\lambda_i \in   \begin{cases}
  \Z^{\geq 0}, & \text{if $ i \leq n-1$} \\
  \frac{1}{2} \Z^{\geq 0}, & \text{ if $i =n$},
\end{cases} \]
then the representation $V_\lambda$ of
 $\Sp_{2n}(\C)$ considered in \cite{BK} has highest weight
 $ 2\lambda_1 \geq 2\lambda_2 \geq \cdots \geq 2\lambda_n$,
 whereas the representation of $\Spin_{2n+1}(\C)$ considered in this paper
 has highest weight $ \lambda_1 \geq \lambda_2 \geq \cdots \geq \lambda_n$.
\end{remark}

Tensor product of two irreducible representations of $\SL_{2n}(\C)$ typically decomposes with high multiplicities.
Here is  a nice class of selfdual representations where multiplicity one holds.

\vspace{5mm}

\begin{example} \label{example1}
  For the representation $V_{\varpi_n}$ of $\SL_{2n}(\C)$,
  \begin{eqnarray*}
  V_{\varpi_n} & = & [0,0,\cdots, 0, 1, 0, \cdots, 0, 0] = \Lambda^n(\C^{2n}), 
  \end{eqnarray*}
  we have
  \[ V_{\varpi_n} \otimes V_{\varpi_n} = V_0 +V_1+ \cdots + V_n, \tag{1}\]
  where
  $V_i$, for $0\leq i \leq n$, are the irreducible selfdual representations of $\SL_{2n}(\C)$ with highest weights
  $\varpi_i + \varpi_{2n-i}$, where we take $\varpi_0=0$. Under the correspondence of irreducible
  selfdual representation of $\SL_{2n}(\C)$ with
  irreducible representations of $\Spin_{2n+1}(\C)$ that we are considering,
the representation $V_i$  of $\SL_{2n}(\C)$ with highest weight
$\varpi_i + \varpi_{2n-i}$ corresponds to the irreducible representation of $\Spin_{2n+1}(\C)$
with highest weight $\varpi_i$ for $i<n$, but for $i=n$,  the representation $V_n$  of $\SL_{2n}(\C)$ with highest weight
$2\varpi_n$ corresponds to the irreducible representation of $\Spin_{2n+1}(\C)$
with highest weight $2\varpi_n$, thus to    the representation $\Lambda^n(V)$ where $V$ is the standard $(2n+1)$ dimensional
representation of $\SO_{2n+1}(\C)$.

Under the correspondence of irreducible
  selfdual representation of $\SL_{2n}(\C)$ with
  irreducible representations of $\Spin_{2n+1}(\C)$ that we are considering,
the representation $V_{\varpi_n}$  of $\SL_{2n}(\C)$ 
  corresponds to the Spinor representation of $\Spin_{2n+1}(\C)$,
  which let's denote by $S_n$. We have the following well-known decomposition for the tensor product of the Spinor
  representation:

  \[ S_n \otimes S_n = V + \Lambda^2(V) + \Lambda^3(V) + \cdots + \Lambda^n(V), \tag{2}\]
  where $V$ is the standard $(2n+1)$ dimensional representation of $\SO_{2n+1}(\C)$, and the
  $\Lambda^i(V)$
  are the fundamental irreducible representations of $\SO_{2n+1}(\C)$ (or, $\Spin_{2n+1}(\C)$)
  of highest weight
  $\varpi_i$ for $i <n$; finally $\Lambda^n(V)$
  is the irreducible representations of $\SO_{2n+1}(\C)$ (or, $\Spin_{2n+1}(\C)$) of highest weight
  $2\varpi_n = (e_1+e_2+\cdots + e_n)$. We thus see that in the decomposition (1) for $\SL_{2n}(\C)$,
  each irreducible component appears with multiplicity 1, and is selfdual,  and the decomposition (1) for $\SL_{2n}(\C)$ and  (2) for 
$\Spin_{2n+1}(\C)$ match perfectly as suggested by Proposition \ref{mult}.
  
\end{example}

\begin{example}
More generally, for the selfdual representations of $\SL_{2n}(\C)$,
\begin{eqnarray*}
  V_a & = & [0,0,\cdots, 0, a, 0, \cdots, 0, 0] {\rm~~ of ~~ highest ~ weight~~} a {\rm~~ times ~ that ~ of~~}
  \Lambda^n(\C^{2n}),  \\
  V_b & = & [0,0,\cdots, 0, b,0, \cdots, 0,0] {\rm~~ of ~~ highest ~ weight~~} b {\rm~~ times ~ that ~ of~~}
  \Lambda^n(\C^{2n}), 
\end{eqnarray*}
(where $a,b$ are positioned at the $n$th place with all other entries zero)  $V_a \otimes V_b$ decomposes as a sum of distinct selfdual irreducible representations of 
$\SL_{2n}(\C)$ each with multiplicity 1. Multiplicity 1, and hence selfduality,
follows from:

\begin{enumerate}

  \item $V_a \otimes V_b$ is realized as space of sections of a line bundle
on $\SL_{2n}(\C)/P_{n,n} \times \SL_{2n}(\C)/P_{n,n}$  where $P_{n,n}$ is the
parabolic in $\SL_{2n}(\C)$ stabilizing an $n$-dimensional subspace of
$\C^{2n}$
 \item $X=\SL_{2n}(\C)/P_{n,n} \times \SL_{2n}(\C)/P_{n,n}$  is a spherical variety
  for $\SL_{2n}(\C)$, i.e., any Borel subgroup in $\SL_{2n}(\C)$ has an open orbit on $X$, for instance because $X$ contains $\GL_{2n}(\C)/  [\GL_n(\C) \times
    \GL_{n}(\C)]$ as an open orbit which is one of the well-known
  spherical varieties for $\SL_{2n}(\C)$.
\end{enumerate}
However, explicit decomposition, as in Example \ref{example1}
seems not clear from this point of view.
The author thanks N. Ressayre for the argument in this example.
\end{example}

\begin{example}
  \label{example:1_0_0}
  For the representation $V_{\varpi_1 +\varpi_{2n-1}}$ of $\SL_{2n}(\C)$,
  \begin{eqnarray*}
    V_{\varpi_1 + \varpi_{2n-1}} & = & [1,0,\cdots, 0,  0, \cdots, 0, 1]
    = {\End} (\C^{2n}) - \C, 
  \end{eqnarray*}
and $W$ any selfdual representation of $\SL_{2n}(\C)$ given by:
  \[ W= [\underline{a}]
  = [a_1,a_2,\cdots,  a_n,  a_{n-1}, \cdots, a_2, a_1], \]
we split the calculation of $ (V_{\varpi_1 + \varpi_{2n-1}}) \otimes W$ into two cases:

  \begin{enumerate}
\item $a_n\not  = 0$. In this case, number of nonzero entries, say $d$, in 
$[\underline{a}]$ is  odd, and 
  $ (V_{\varpi_1 + \varpi_{2n-1}}) \otimes W$
  contains $d$ distinct selfdual representations of $\SL_{2n}(\C)$ which are not $W$ each with multiplicity 1, and
    $W$ itself appears with multiplicity $d$. This follows by calculating tensor product
$ V_{\varpi_1}  \otimes W$ and then $ V_{\varpi_{2n-1}} \otimes ( V_{\varpi_1} \otimes W) $, and noting that 
    $$V_{\varpi_1 + \varpi_{2n-1}} =  {\End} (\C^{2n}) - \C =    V_{\varpi_1} \otimes     V_{\varpi_{2n-1}} - \C.$$

  \item  $a_n = 0$. In this case, number of nonzero entries, say $d$, in 
$[\underline{a}]$ is  even, and 
    $ (V_{\varpi_1 + \varpi_{2n-1}}) \otimes W$
    contains $d$ selfdual representations which are not $W$, and
    $W$ itself appears with multiplicity $d$, by an argument similar to the one used in the first case.

  \end{enumerate}

  In case 1, since the multiplicity of all selfdual constituents of $ (V_{\varpi_1 + \varpi_{2n-1}}) \otimes W$ is odd,
  they appear in the corresponding tensor product $ (V_{\varpi_1} ) \otimes W'$  of the $\Spin_{2n+1}(\C)$ where
    $W'= [a_1,a_2,\cdots,  a_n]$.

    In case 2, the representation $W$ of $\Sp_{2n}(\C)$ appears with even multiplicity in $ (V_{\varpi_1 + \varpi_{2n-1}}) \otimes W$, but
    the corresponding tensor product $ V_{\varpi_1}\otimes W'$
    of the $\Spin_{2n+1}(\C)$ where
    $W'= [a_1,a_2,\cdots,  a_n]$ does not contain $W'$. Because the zero weight space of 
    $ V_{\varpi_1}$ has dimension 1, $ V_{\varpi_1}\otimes W'$ contains $W'$ with multiplicity $\leq 1$. But by
    Proposition \ref{mult}, multiplicity of $W'$ in $ V_{\varpi_1}\otimes W'$
    has the same parity as the multiplicity of the representation $W$ of
$\SL_{2n}(\C)$
    in $ (V_{\varpi_1 + \varpi_{2n-1}}) \otimes W$ which we have analyzed above to be even. Thus the
    multiplicity of $W'$ in $ V_{\varpi_1}\otimes W'$ must be zero. (This must surely have a direct proof!) 
  \end{example}

\section{The pair $(\SL_{2n+1}(\C), \Sp_{2n}(\C))$}

In this section  we will focus attention on the pair 
$(G,H)=(\SL_{2n+1}(\C),
\Sp_{2n}(\C))$ which are related to each other through fixed points of a diagram automorphism
on the bigger group, to be more precise, the $L$-group of $H$  has an embedding into the $L$-group of $G$,
with image which is the fixed points of a diagram automorphism on the $L$-group of $G$, see \cite{KLP} for the precise statements.

We will continue to follow the notations of the {\sf LiE} software for denoting
representations of $\SL_{2n+1}(\C)$ and $ \Sp_{2n}(\C))$.

Let $\varpi_i$ denote the highest weights of the fundamental representations
of $\SL_{2n+1}(\C)$
realized on the $i$-th exterior powers of the standard representation $V=\C^{2n+1}$, thus:
\begin{eqnarray*}
\varpi_1 & = &  e_1 \\
  \varpi_2 & = &  e_1 +e_2\\
  \varpi_3 & = &  e_1 +e_2+e_3\\
  & \cdot & \\
  \varpi_i & = &  e_1+e_2+\cdots+ e_i \\
  & \cdot & \\
  \varpi_{2n} & = &  e_1+e_2+\cdots +e_{2n}.
\end{eqnarray*}

The {\sf LiE} software denotes an irreducible representation of $\SL_{2n+1}(\C)$
by $[a_1,a_2,\cdots, a_{2n}]$ with $a_i$ integers $\geq 0$, which stands for the
irreducible representation of $\SL_{2n+1}(\C)$ with highest weight
$$a_1\varpi_1 + a_2\varpi_2+ \cdots + a_{2n}\varpi_{2n}.$$

Since in our work we are looking only at selfdual representations, we will have:

 $$[a_1,a_2,\cdots, a_{2n}] = [a_{2n},a_{2n-1},\cdots, a_{1}], $$
thus,
 $$[a_1,a_2,\cdots, a_{2n}] = [a_{1},a_{2},\cdots,a_{n-1},  a_n, a_n, a_{n-1}, \cdots, a_2, a_{1}]. $$

Next we come to the symplectic group $\Sp_{2n}(\C)$. Let $\varpi_i$, $1 \leq i \leq (n-1)$ denote the highest weights of the fundamental representations realized on a subspace of the  $i$-th exterior powers of the standard representation $V=\C^{2n}$ of
$\Sp_{2n}(\C)$, thus:
\begin{eqnarray*}
\varpi_1 & = &  e_1 \\
  \varpi_2 & = &  e_1 +e_2\\
  \varpi_3 & = &  e_1 +e_2+e_3\\
  & \cdot & \\
  \varpi_i & = &  e_1+e_2+\cdots + e_i \\
  & \cdot & \\
\varpi_{n-1} & = &  e_1+e_2+\cdots + e_{n-1}\\
  \varpi_n &=& (e_1+e_2+\cdots +e_n).
\end{eqnarray*}

The {\sf LiE} software denotes an irreducible representation
of $\Sp_{2n}(\C)$ by $[a_1,a_2,\cdots, a_{n}]$ with $a_i$ integers $\geq 0$, which stands for the irreducible representation
of $\Sp_{2n}(\C)$ with highest weight
$$a_1\varpi_1 + a_2\varpi_2+ \cdots + a_{n}\varpi_{n}.$$

We note the following lemma
whose straightforward proof will be omitted.

\begin{lemma} \label{central2}
  A selfdual representation  $[a_{1},a_{2},\cdots,  a_n, a_n, \cdots, a_2, a_{1}] $
  of $\SL_{2n+1}(\C)$ has trivial central character.
  
  The representation  $[a_{1},a_{2},\cdots,a_{n-1},  a_n] $ of $\Sp_{2n}(\C)$ has central character
  of order $\leq 2$, and has trivial central character if and only if $ a_1+ a_3 + a_5 + \cdots $ is even.
  \end{lemma}

\begin{conj}
Recall first that by Proposition \ref{mult}, if for three irreducible representations of $\Sp_{2n}(\C)$ given by:
\begin{eqnarray*}
  \pi_1 & = & [a_1,a_2,\cdots, a_{n}], \\
  \pi_2 & = & [b_1,b_2,\cdots, b_{n}], \\
  \pi_3 & = & [c_1,c_2,\cdots, c_{n}],
\end{eqnarray*}
with corresponding
selfdual representations of $\SL_{2n+1}(\C)$,
\begin{eqnarray*}
  V_1 & = & [a_1,a_2,\cdots, a_{n-1}, a_{n}, a_n, a_{n-1}, \cdots, a_2,a_1] \\
  V_2 & = & [b_1,b_2,\cdots, b_{n-1}, b_{n},b_n, b_{n-1}, \cdots, b_2,b_1], \\
  V_3 & = & [c_1,c_2,\cdots, c_{n-1},c_{n},c_n, c_{n-1}, \cdots, c_2,c_1]
\end{eqnarray*}
$$ 1_{\Sp_{2n}(\C)} \subset \pi_1 \otimes \pi_2 \otimes \pi_3 \implies 1_{\SL_{2n+1}(\C)} \subset V_1 \otimes V_2 \otimes V_3.$$

We propose that if at least two of the three $a_1,b_1,c_1$ are nonzero, and the central character of
the representation $\pi_1 \otimes \pi_2 \otimes \pi_3$ of $\Sp_{2n}(\C)$ is trivial,
then the converse holds:

$$ 1_{\SL_{2n+1}(\C)}
\subset V_1 \otimes V_2 \otimes V_3
\implies 1_{\Sp_{2n}(\C)}
\subset \pi_1 \otimes \pi_2 \otimes \pi_3.$$ 
\end{conj}

We note the following curious proposition.

\begin{prop}
  If the central character of the representation $\pi_1 \otimes \pi_2 \otimes \pi_3,$ (given by Lemma \ref{central2})
  of $\Sp_{2n}(\C)$ is nontrivial,
  then for the corresponding representation $V_1 \otimes V_2 \otimes V_3$ of $\SL_{2n+1}(\C)$,
  $$m_{\SL_{2n+1}(\C)}[V_1, V_2\otimes V_3] \in 2 \Z.$$
\end{prop}
\begin{proof}
  Because of central character consideration,  $m_{\Sp_{2n}(\C)}[\pi_1, \pi_2 \otimes \pi_3] =0$. On the other hand, by
  Corollary \ref{mult-cor},
  $$  m_{\SL_{2n+1}(\C)}[V_1, V_2\otimes V_3] \equiv m_{\Sp_{2n}(\C)}[\pi_1, \pi_2 \otimes \pi_3] \bmod 2.  $$
  The proof of the proposition is therefore clear.
  \end{proof}

\section{Some questions relating $B_n, C_n$ and $A_{2n}, A_{2n+1}$}

In many ways although the groups $\Spin_{2n+1}(\C)$ and $\Sp_{2n}(\C)$
are similar, their tensor products tends to be quite different. However, we noticed experimentally an interesting class of examples where the tensor products agree. We pose this as a
question. In this question, we parametrize representations of both the groups --- just as before ---
by an $n$-tuple of non-negative integers
$\underline{a}
= [a_1,a_2,\cdots, a_n]$, denoted as $V_{\underline{a}}$ for $\Spin_{2n+1}(\C)$ and $W_{\underline{a}}$ for 
$\Sp_{2n}(\C)$ as done in the last section.

\vspace{5mm}

\begin{question}
    For integers $k \geq 0$, let $V_k = [k,0,0,\cdots, 0]$ be the irreducible representations of  
$\Spin_{2n+1}(\C)$, and $W_k = [k,0,0,\cdots, 0]$ be the  irreducible representations of  
$\Sp_{2n}(\C)$. Then for any two non-negative integers $k,\ell$,
both the representations $V_k \otimes V_\ell$
of $\Spin_{2n+1}(\C)$
and the representations $W_k \otimes W_\ell$
of $\Sp_{2n}(\C)$ decompose with multiplicity 1. Further,
\end{question}

\begin{enumerate}

\item If $n>2$, the irreducible constituents of the two tensor products
are the same, i.e., $V_{\underline{a}} \subset V_k \otimes V_\ell$ if and only if
$W_{\underline{a}} \subset W_k \otimes W_\ell$.

\item If $n=2$, an irreducible representation $V_{\underline{a}}$ for ${\underline{a}} = [a_1,a_2]$
   appears in $ V_k \otimes V_\ell$ if and only if
   $W_{\underline{a}}$ for ${\underline{a}}
   = [a_1, 2a_2] $ appears in  $ W_k \otimes W_\ell$
   (and the second co-ordinate of $\underline{a}$ for any
   $W_{\underline{a}} \subset W_k \otimes W_\ell$ is even).
   \end{enumerate}
\vspace{5mm}

The assertion that the representations $W_k \otimes W_\ell$
of $\Sp_{2n}(\C)$ decompose with multiplicity 1 is a simple consequence of the fact that for the
natural action of $\Sp_{2n}(\C)$ on $\C^{2n} \oplus \C^{2n}$,
for any Borel subgroup
$ B \subset G = \Sp_{2n}(\C) \times (\C^\times \times \C^\times)$,
where the two $\C^\times$ correspond to scaling on the two copies of $\C^{2n}$ (inside $\C^{2n} \oplus \C^{2n}$),
$B$ has an open orbit on $\C^{2n} \oplus \C^{2n}$. But this argument does not
apply, at least immediately,
to prove multiplicity 1 for $V_k \otimes V_\ell$; however, it is one of the multiplicity 1 pairs of Stembridge in \cite{St}.
\vspace{5mm}

The above question is closely related to the following question on special linear groups.

\begin{question}
For integers $k \geq 0$, let $V_k = [k,0,0,\cdots, 0, 0, k]$ be the irreducible selfdual representation of  
$\SL_{2n}(\C)$, and $W_k = [k,0,0,\cdots, 0,0, k]$ be the  irreducible selfdual representation of  
$\SL_{2n+1}(\C)$. Then for any two non-negative integers $k,\ell$,
irreducible selfdual representations of $\SL_{2n}(\C)$ which appear with odd multiplicity
inside the representation $V_k \otimes V_\ell$
of $\SL_{2n}(\C)$ are in bijective correspondence with
irreducible selfdual representations of $\SL_{2n+1}(\C)$ which appear with odd multiplicity
inside the representation $W_k \otimes W_\ell$
of $\SL_{2n+1}(\C)$ given by the correspondence:

$$[a_1,a_2,\cdots, a_n, a_{n-1}, \cdots, a_1] \longleftrightarrow [a_1,a_2,\cdots, a_n, a_n, a_{n-1}, \cdots, a_1],$$
if $n>2$, in which case all $a_i, i>2$ are zero. If $n=2$, then 
$$[a_1,2a_2, a_1] \longleftrightarrow [a_1,a_2, a_2, a_1].$$

Further, under the above correspondence, multiplicities (we are only looking at odd multiplicities) are preserved.

\end{question}

\begin{remark}
  We note that multiplicity free tensor products (of $\sigma$-invariant representations)  for $G$ are especially useful for our purposes,
  since they exactly reflect a similar
  decomposition of the tensor product for $G_\sigma$. Thus one is led to ask for 
  any simple group $G$, all pairs of irreducible representations $V,W$ of $G$ such that $V \otimes W$ decomposes with multiplicity 1. In fact this question has been
  fully answered by J. Stembridge in \cite{St}, and the answer is closely related to a classical problem of classifying spaces of the form $G/P_1 \times G/P_2$
  which are spherical,i.e., $B$ has an open orbit on it. A stronger question 
  of classifying representations whose restriction to a maximal torus decomposes with multiplicity $\leq 1$ has also been known, see the introduction
  of the paper \cite{St} for a very small list (due to R. Howe in \cite{Ho}, Theorem 4.6.3.); one can also use the work \cite{BZ} to give a proof
  of the theorem of Howe.
\end{remark}

\section{Sample computations}
\label{sec:observations}

\subsection*{Tables representing sample computations}
\label{sec:tables}

We use the following notations in the two tables below (Table \ref{tab:tableA3}  
and \ref{tab:tableA5} ).
\begin{itemize}
\item The row and column headers
  $\{a_1,a_2\}$ in Table 1 represent the irreducible selfdual representations of
  $\SL_4(\C)$ which were earlier represented by $[a_1,a_2,a_1] = a_1\varpi_1 + a_2\varpi_2 + a_{1}\varpi_{3},$
  with $\varpi_i$ the $i$-th fundamental weight.
  Similarly, the row and column headers  ${\{a_1,a_2,a_3\}}$ in Table 2 represent the selfdual
  representations $[a_1,a_2,a_3,a_2,a_1] = a_1\varpi_1 + a_2\varpi_2+ \cdots + a_{1}\varpi_{5}$ of $\SL_6(\C)$.
  
\item 
  For the cell appearing in the row corresponding to a representation
  ${V}$ and column corresponding to a representation ${W}$, there
  are 4 numbers: $\begin{matrix} n_1&n_4\\n_2&n_3
  \end{matrix}$, where
  \begin{itemize}
  \item[$n_1$:] number of representations of $\SL_{2n}(\C)$ appearing in the tensor product 
    ${V} \otimes {W}$.
  \item[$n_2$:] number of selfdual representations of $\SL_{2n}(\C)$ appearing
    in the tensor product ${V} \otimes {W}$.
  \item[$n_3$:] number of representations of $\Spin_{2n+1}(\C)$ appearing in the
    tensor product of the representations of $\Spin_{2n+1}(\C)$ corresponding to
    ${V'}$ and ${W'}$.
  \item[$n_4$:] number of selfdual representations of $\SL_{2n}(\C)$ in the tensor product
    ${V}\otimes {W}$, such that the corresponding
    representation of $\Spin_{2n+1}(\C)$ is not appearing in the tensor
    product of the corresponding representations.
  \end{itemize}
\end{itemize}

\newpage 

\begin{table}[!htbp]
  \centering
\makebox[\linewidth]
  {\footnotesize
     \begin{tabular}{|r||cc|cc|cc|cc|cc|cc|cc|cc|cc|cc|}
       \hline
       &\multicolumn{2}{c|}{\bf \{1,0\}} & \multicolumn{2}{c|}{\bf \{1,1\}} & \multicolumn{2}{c|}{\bf \{2,2\}} & \multicolumn{2}{c|}{\bf \{2,8\}} & \multicolumn{2}{c|}{\bf \{5,0\}} & \multicolumn{2}{c|}{\bf \{5,9\}} & \multicolumn{2}{c|}{\bf \{8,0\}} \\ \hline 
       \multirow{2}{*}{\bf \{0,1\}} & 4 & 0 & 6 & 0 & 6 & 0 & 6 & 0 & 4 & 0 & 6 & 0 & 4 & 0 \\ 
       & 2 & 2 & 4 & 4 & 4 & 4 & 4 & 4 & 2 & 2 & 4 & 4 & 2 & 2 \\ \hline 
       \multirow{2}{*}{\bf \{0,8\}} & 5 & 0 & 13 & 0 & 55 & 0 & 202 & 0 & 87 & 0 & 395 & 0 & 165 & 0 \\ 
       & 3 & 3 & 7 & 7 & 19 & 19 & 60 & 60 & 19 & 19 & 75 & 75 & 25 & 25 \\ \hline 
       \multirow{2}{*}{\bf \{1,2\}} & 11 & 0 & 22 & 0 & 48 & 0 & 59 & 0 & 43 & 1 & 79 & 0 & 43 & 1 \\ 
       & 5 & 5 & 8 & 8 & 16 & 16 & 19 & 19 & 11 & 10 & 21 & 21 & 11 & 10 \\ \hline 
       \multirow{2}{*}{\bf \{1,6\}} & 11 & 0 & 24 & 0 & 87 & 0 & 229 & 0 & 138 & 1 & 398 & 0 & 217 & 1 \\ 
       & 5 & 5 & 10 & 10 & 25 & 25 & 57 & 57 & 26 & 25 & 72 & 72 & 31 & 30 \\ \hline 
       \multirow{2}{*}{\bf \{2,0\}} & 8 & 1 & 18 & 0 & 40 & 1 & 45 & 0 & 27 & 3 & 55 & 0 & 27 & 3 \\ 
       & 4 & 3 & 6 & 6 & 12 & 11 & 13 & 13 & 9 & 6 & 13 & 13 & 9 & 6 \\ \hline 
       \multirow{2}{*}{\bf \{2,4\}} & 13 & 0 & 34 & 0 & 109 & 0 & 229 & 0 & 164 & 2 & 386 & 0 & 229 & 2 \\ 
       & 5 & 5 & 12 & 12 & 27 & 27 & 51 & 51 & 30 & 28 & 66 & 66 & 33 & 31 \\ \hline 
       \multirow{2}{*}{\bf \{2,5\}} & 13 & 0 & 34 & 0 & 117 & 0 & 291 & 0 & 189 & 0 & 499 & 0 & 300 & 0 \\ 
       & 5 & 5 & 12 & 12 & 27 & 27 & 57 & 57 & 31 & 31 & 81 & 81 & 38 & 38 \\ \hline 
       \multirow{2}{*}{\bf \{3,0\}} & 8 & 1 & 22 & 0 & 63 & 2 & 91 & 0 & 62 & 6 & 145 & 0 & 64 & 6 \\ 
       & 4 & 3 & 6 & 6 & 17 & 15 & 21 & 21 & 16 & 10 & 25 & 25 & 16 & 10 \\ \hline 
       \multirow{2}{*}{\bf \{3,3\}} & 13 & 0 & 38 & 0 & 126 & 0 & 272 & 0 & 194 & 0 & 470 & 0 & 282 & 0 \\ 
       & 5 & 5 & 12 & 12 & 26 & 26 & 50 & 50 & 32 & 32 & 72 & 72 & 36 & 36 \\ \hline 
       \multirow{2}{*}{\bf \{7,0\}} & 8 & 1 & 22 & 0 & 104 & 2 & 393 & 2 & 196 & 15 & 829 & 0 & 400 & 28 \\ 
       & 4 & 3 & 6 & 6 & 20 & 18 & 53 & 51 & 36 & 21 & 87 & 87 & 64 & 36 \\ \hline 
     \end{tabular}
  }
  \vskip2mm
  \caption{$A_3-B_2$ sample computations}
  \label{tab:tableA3}
\end{table}

\subsubsection*{Comments on table \ref{tab:tableA3}}
\label{sec:table_A3_comments}

\begin{itemize}
\item Observe that $n_4$ is often zero in this table as predicted by Conjecture
  1. For example, the column corresponding to $\{5,9\}$ has $n_4$ identically
  zero by Remark \ref{odd}. In fact Remark \ref{odd} allows us to showcase many
  pairs of representations of $\SL_4(\C)$ for which $n_4$ is zero, but our work
  has tried to uncover how often $n_4$ is nonzero, and the reader will agree
  with us by looking at this table that almost all entries in the table have
  either $n_4$ equal to zero, or (comparatively) very small value for $n_4$
  except for the bottom corner entry corresponding to tensor product:
  $$[8,0,8]  \otimes [7,0,7],$$
  of irreducible selfdual representations of $\SL_4(\C)$ which decomposes with
  400 irreducible representations (with various multiplicities not discussed in
  the table), out of which only 64 are selfdual. The corresponding tensor
  product $[8,0] \otimes [7,0]$ of $\Spin_5(\C)$ has 36 entries, thus there are
  $n_4=28$ many irreducible selfdual representations of $\SL_4(\C)$ which go
  `missing' in $\Spin_5(\C)$.
\end{itemize}

\newpage 

\begin{table}[htbp]
  \centering
\makebox[\linewidth]
  {\footnotesize
     \begin{tabular}{|r||cc|cc|cc|cc|cc|cc|cc|cc|cc|cc|}
       \hline
       &\multicolumn{2}{c|}{\bf \{0,0,4\}} & \multicolumn{2}{c|}{\bf \{1,0,0\}} & \multicolumn{2}{c|}{\bf \{2,2,2\}} & \multicolumn{2}{c|}{\bf \{5,1,0\}} & \multicolumn{2}{c|}{\bf \{5,9,9\}} & \multicolumn{2}{c|}{\bf \{8,5,0\}} & \multicolumn{2}{c|}{\bf \{8,7,0\}} \\ \hline 
       \multirow{2}{*}{\bf \{0,0,8\}} & 35 & 0 & 5 & 0 & 2101 & 0 & 734 & 0 & 27193 & 0 & 10907 & 0 & 15096 & 0 \\ 
       & 35 & 35 & 3 & 3 & 143 & 143 & 62 & 62 & 675 & 675 & 271 & 271 & 316 & 316 \\ \hline 
       \multirow{2}{*}{\bf \{0,6,2\}} & 409 & 0 & 15 & 0 & 8451 & 0 & 5581 & 4 & 184421 & 0 & 60745 & 14 & 91263 & 20 \\ 
       & 57 & 57 & 5 & 5 & 289 & 289 & 191 & 187 & 1837 & 1837 & 891 & 877 & 1153 & 1133 \\ \hline 
       \multirow{2}{*}{\bf \{1,5,5\}} & 907 & 0 & 29 & 0 & 13664 & 0 & 10179 & 0 & 314922 & 0 & 96685 & 0 & 148136 & 0 \\ 
       & 95 & 95 & 7 & 7 & 370 & 370 & 261 & 261 & 2674 & 2674 & 1105 & 1105 & 1436 & 1436 \\ \hline 
       \multirow{2}{*}{\bf \{2,4,1\}} & 935 & 0 & 30 & 0 & 7362 & 0 & 5975 & 0 & 129494 & 0 & 45822 & 0 & 68068 & 0 \\ 
       & 63 & 63 & 6 & 6 & 236 & 236 & 183 & 183 & 1386 & 1386 & 626 & 626 & 778 & 778 \\ \hline 
       \multirow{2}{*}{\bf \{2,5,8\}} & 1290 & 0 & 31 & 0 & 20874 & 0 & 16476 & 0 & 588058 & 0 & 175590 & 5 & 267962 & 9 \\ 
       & 110 & 110 & 7 & 7 & 526 & 526 & 368 & 368 & 3800 & 3800 & 1702 & 1697 & 2234 & 2225 \\ \hline 
       \multirow{2}{*}{\bf \{2,6,0\}} & 700 & 0 & 22 & 1 & 10898 & 4 & 7899 & 37 & 245139 & 0 & 74960 & 118 & 116170 & 145 \\ 
       & 54 & 54 & 6 & 5 & 310 & 306 & 235 & 198 & 1977 & 1977 & 978 & 860 & 1262 & 1117 \\ \hline 
       \multirow{2}{*}{\bf \{3,0,2\}} & 212 & 0 & 15 & 0 & 1921 & 0 & 1170 & 3 & 11143 & 0 & 5401 & 0 & 5923 & 0 \\ 
       & 36 & 36 & 5 & 5 & 121 & 121 & 80 & 77 & 305 & 305 & 165 & 165 & 165 & 165 \\ \hline 
       \multirow{2}{*}{\bf \{4,3,1\}} & 1090 & 0 & 30 & 0 & 8732 & 0 & 7242 & 0 & 164516 & 0 & 55438 & 0 & 84360 & 0 \\ 
       & 66 & 66 & 6 & 6 & 252 & 252 & 202 & 202 & 1506 & 1506 & 680 & 680 & 846 & 846 \\ \hline 
       \multirow{2}{*}{\bf \{4,4,0\}} & 925 & 0 & 22 & 1 & 9977 & 7 & 7667 & 40 & 200069 & 0 & 64022 & 105 & 99316 & 116 \\ 
       & 59 & 59 & 6 & 5 & 303 & 296 & 241 & 201 & 1609 & 1609 & 840 & 735 & 1040 & 924 \\ \hline 
       \multirow{2}{*}{\bf \{9,3,0\}} & 760 & 0 & 22 & 1 & 12922 & 10 & 9488 & 46 & 420994 & 0 & 110432 & 145 & 178101 & 155 \\ 
       & 54 & 54 & 6 & 5 & 310 & 300 & 242 & 196 & 2484 & 2484 & 1228 & 1083 & 1569 & 1414 \\ \hline 
       \multirow{2}{*}{\bf \{9,8,7\}} & 1751 & 0 & 31 & 0 & 58765 & 0 & 44816 & 0 & 2128066 & 0 & 743281 & 0 & 1084932 & 0 \\ 
       & 125 & 125 & 7 & 7 & 777 & 777 & 526 & 526 & 8242 & 8242 & 3549 & 3549 & 4642 & 4642 \\ \hline 
     \end{tabular}
  }
  \vskip2mm
  \caption{$A_5-B_3$ sample computations}
  \label{tab:tableA5}
\end{table}

\subsubsection*{Comments on table \ref{tab:tableA5}}
\label{sec:table_A5_comments}
\begin{itemize}

\item Observe that just as in table 1, $n_4$ is often zero in this table too
  as predicted by Conjecture
  1. For example, the columns corresponding to $\{5,9,9\}$ and $\{9,8,7\}$ has
  $n_4$ identically
  zero by Remark \ref{odd}. In fact if the last entry is odd, either in a row or in a column,
  then $n_4$ identically
  zero by Remark \ref{odd}.
  
\item In this table too we find that 
  either $n_4$ equal to zero, or (comparatively) very small value for $n_4$
  except for the bottom corner entry corresponding to tensor product:
  $$[8,7,0,7,8]  \otimes [9,3,0,3,9],$$
  of irreducible selfdual representations of $\SL_6(\C)$ which decomposes with
  178101 irreducible representations (with various multiplicities not discussed in
  the table), out of which only 1569 are selfdual. The corresponding tensor
  product $[8,7,0] \otimes [9,3,0]$ of $\Spin_7(\C)$ has 1414 entries, thus there are
  $n_4=155$ many irreducible selfdual representations of $\SL_4(\C)$ which go
  `missing' in $\Spin_7(\C)$. There are similar conclusions of 
  for representations  $V=[a_1,a_2,a_3,a_2,a_1]$ and $W=[b_1,b_2,b_3,b_2,b_1]$
  of $\SL_6(\C)$ when for both, the middle terms $a_3=b_3=0$.
\item If $v=[1,0,0,0,1]$, then for any $w=[b_1,b_2,b_3,b_2,b_1]$ there is
  exactly one representation is missing if and only if $b_3=0$. In such case,
  the missing representation is necessarily $w$. This is a special case of the
  Example \ref{example:1_0_0} above.
  
\end{itemize}

\subsubsection*{Special case of representations of the form $[m,0,0,0,m]$}
\label{sec:special_case_m-0-0}

Here, we consider the tensor product of the representations of the form
$V_m=[m,0,0,0,m]$ and $V_n=[n,0,0,0,n]$ of $\SL_6(\C)$. Here is a summary of some
of the observations made regarding  the tensor product
$V_m\otimes V_n$ of $\SL_6(\C)$ and the  corresponding tensor product of the representations
$W_m=[m,0,0]$ and $W_n=[n,0,0]$  of  $\Spin_7(\C)$.

\begin{itemize}
\item total number of distinct irreducible representations in $V_m \otimes V_n$
  is at most $(m+1)^3$, for all $n\leq 2m$. Further, for $n=m$, this number is
  equal to $\frac{m\left(2m^2 + 1\right)}{3}$.
\item total number of distinct selfdual representations in $V_m \otimes V_n$ =
  $(n+1)^2$, for all $n\leq m$. 
\item total number of distinct selfdual representations in $V_m \otimes V_n$
  missing in the corresponding tensor product in $\Spin_7(\C)$ is =
  $\frac{n(n+1)}{2}$ for all $n\leq m$. 
\item $V_p \subset V_m \otimes V_n$ (where $V_i=[i,0,0,0,i]$ is an irreducible
  representation of $\SL_6(\C)$) if and only if $ m-n \leq p \leq m+n$.
\item $W_p \subset W_m \otimes W_n$ (where $W_i=[i,0,0]$ is an irreducible
  representation of $\Spin_7(\C)$) if and only if $ m-n \leq p \leq m+n$ and $p$
  has the same parity as $m+n$.

\end{itemize}
  
\vspace{5mm}

Our next table is the summary of some results about the tensor products for the
selfdual irreducible representations of $\SL_5(\C)$ using the notation
$\{a_1,a_2\} = [a_1,a_2,a_2,a_1] = a_1\varpi_1 + a_2\varpi_2+ a_2\varpi_3 +
a_{1}\varpi_{4}$, otherwise we use the same notation as in earlier tables,
except with the important difference that now $n_2$ represents the number of
selfdual representations of $\SL_5(\C)$ represented say by $[c_1,c_2,c_2,c_1]$
with $c_1$ of the same parity as that of $a_1+b_1$ which is the central
character condition which appears in Lemma \ref{central2}.  The table now
compares tensor products of irreducible selfdual representations of $\SL_5(\C)$
with those of $\Sp_4(\C)$.

\begin{table}[htbp]
  \centering
\makebox[\linewidth]
  {\footnotesize
     \begin{tabular}{|r||cc|cc|cc|cc|cc|cc|cc|cc|cc|cc|}
       \hline
       &\multicolumn{2}{c|}{\bf \{0,1\}} & \multicolumn{2}{c|}{\bf \{0,4\}} & \multicolumn{2}{c|}{\bf \{0,7\}} & \multicolumn{2}{c|}{\bf \{5,0\}} & \multicolumn{2}{c|}{\bf \{5,9\}} & \multicolumn{2}{c|}{\bf \{6,2\}} & \multicolumn{2}{c|}{\bf \{8,0\}} & \multicolumn{2}{c|}{\bf \{8,9\}} \\ \hline 
       \multirow{2}{*}{\bf \{0,1\}} & 12 & 1 & 20 & 1 & 20 & 1 & 14 & 0 & 51 & 0 & 51 & 0 & 14 & 0 & 51 & 0 \\ 
       & 4 & 3 & 4 & 3 & 4 & 3 & 3 & 3 & 5 & 5 & 5 & 5 & 3 & 3 & 5 & 5 \\ \hline 
       \multirow{2}{*}{\bf \{0,5\}} & 20 & 1 & 725 & 10 & 1776 & 15 & 349 & 0 & 6268 & 0 & 2214 & 2 & 883 & 0 & 8092 & 1 \\ 
       & 4 & 3 & 25 & 15 & 36 & 21 & 12 & 12 & 52 & 52 & 36 & 34 & 19 & 19 & 60 & 59 \\ \hline 
       \multirow{2}{*}{\bf \{0,7\}} & 20 & 1 & 985 & 10 & 3648 & 28 & 419 & 0 & 14462 & 0 & 3706 & 2 & 1435 & 0 & 19676 & 3 \\ 
       & 4 & 3 & 25 & 15 & 64 & 36 & 12 & 12 & 88 & 88 & 45 & 43 & 24 & 24 & 104 & 101 \\ \hline 
       \multirow{2}{*}{\bf \{0,9\}} & 20 & 1 & 1035 & 10 & 5342 & 28 & 439 & 0 & 24562 & 0 & 4974 & 2 & 1765 & 0 & 34396 & 5 \\ 
       & 4 & 3 & 25 & 15 & 64 & 36 & 12 & 12 & 132 & 132 & 48 & 46 & 25 & 25 & 156 & 151 \\ \hline 
       \multirow{2}{*}{\bf \{2,0\}} & 14 & 0 & 36 & 0 & 36 & 0 & 27 & 0 & 131 & 0 & 117 & 0 & 27 & 0 & 131 & 0 \\ 
       & 3 & 3 & 4 & 4 & 4 & 4 & 6 & 6 & 9 & 9 & 9 & 9 & 6 & 6 & 9 & 9 \\ \hline 
       \multirow{2}{*}{\bf \{2,4\}} & 51 & 0 & 1210 & 3 & 2872 & 4 & 881 & 0 & 9544 & 0 & 3129 & 0 & 1764 & 0 & 12748 & 0 \\ 
       & 5 & 5 & 32 & 29 & 44 & 40 & 21 & 21 & 66 & 66 & 47 & 47 & 32 & 32 & 78 & 78 \\ \hline 
       \multirow{2}{*}{\bf \{2,5\}} & 51 & 0 & 1514 & 3 & 4111 & 5 & 1072 & 0 & 14051 & 0 & 4169 & 0 & 2404 & 0 & 19181 & 0 \\ 
       & 5 & 5 & 32 & 29 & 59 & 54 & 22 & 22 & 85 & 85 & 53 & 53 & 36 & 36 & 101 & 101 \\ \hline 
       \multirow{2}{*}{\bf \{3,0\}} & 14 & 0 & 92 & 0 & 100 & 0 & 62 & 0 & 469 & 0 & 332 & 0 & 64 & 0 & 471 & 0 \\ 
       & 3 & 3 & 6 & 6 & 6 & 6 & 10 & 10 & 16 & 16 & 15 & 15 & 10 & 10 & 16 & 16 \\ \hline 
       \multirow{2}{*}{\bf \{8,5\}} & 51 & 0 & 3239 & 0 & 11894 & 3 & 2427 & 0 & 42336 & 0 & 12161 & 0 & 8168 & 0 & 60498 & 0 \\ 
       & 5 & 5 & 41 & 41 & 87 & 84 & 36 & 36 & 166 & 166 & 90 & 90 & 75 & 75 & 206 & 206 \\ \hline 
       \multirow{2}{*}{\bf \{9,7\}} & 51 & 0 & 3829 & 0 & 17522 & 0 & 2697 & 0 & 64425 & 0 & 17995 & 0 & 11628 & 0 & 90169 & 0 \\ 
       & 5 & 5 & 41 & 41 & 104 & 104 & 36 & 36 & 221 & 221 & 99 & 99 & 80 & 80 & 251 & 251 \\ \hline 
     \end{tabular}
  }
  \vskip2mm
  \caption{$A_4-C_2$ sample computations}
  \label{tab:tableA4}
\end{table}

\subsubsection*{Comments on table \ref{tab:tableA4}}
\label{sec:table_A4_comments}
\begin{itemize}

\item Observe that just as in table 1, and table 2, $n_4$ is often zero in this table too
  as predicted by Conjecture
  3. For example, the columns corresponding to $\{5,9\}$ or row corresponding to $\{9,7\}$
  has $n_4$ identically
  zero by Remark \ref{odd}. In fact if the first entry is odd, either in a row or in a column,
  then $n_4$ identically
  zero by Remark \ref{odd}.
  
\item In this table too we find that 
  either $n_4$ equal to zero, or (comparatively) very small value for $n_4$
  except for the bottom corner entry corresponding to tensor product where
  both row and column vector have first entry 0, for example
  $$[0,7,7,0]  \otimes [0,9,9,0],$$
  of irreducible selfdual representations of $\SL_5(\C)$ which decomposes with
  5342 irreducible representations (with various multiplicities not discussed in
  the table), out of which only 64 are selfdual. The corresponding tensor
  product $[0,7] \otimes [0,9]$ of $\Sp_4(\C)$ has 28 entries, thus there are
  $n_4=28$ many irreducible selfdual representations of $\SL_4(\C)$ which go
  `missing' in $\Spin_7(\C)$. There are similar conclusions of 
  for representations  $V=[a_1,a_2,a_2,a_1]$ and $W=[b_1,b_2,b_2,b_1]$
  of $\SL_5(\C)$ when for both, the first terms $a_1=b_1=0$.
\end{itemize}

\vspace{1cm}

\noindent{\bf Acknowledgement:} The authors thank CS Rajan for bringing the
paper of Kapovitch \cite{Kap} and that of S. Kumar \cite{Ku2} to their notice.
The authors also  thank J. Hong, S. Kumar and N. Ressayre for very useful comments on this work.


\begin{thebibliography}{MVW} 

\bibitem[BK]{BK} P. Belkale and S. Kumar, {\em  Eigencone, saturation and Horn problems for symplectic and odd orthogonal groups.}
   J. Algebraic Geom. 19 (2010), no. 2, 199-242.

 \bibitem[BZ]{BZ}  A. Berenshtein and  A. Zelevinskii, {\em
   When is the multiplicity of a weight equal to 1?}
   (Russian) Funktsional. Anal. i Prilozhen. 24 (1990), no. 4, 1-13, 96;
   translation in Funct. Anal. Appl. 24 (1990), no. 4, 259–269 (1991).
   
 \bibitem[HS]{HS} J. Hong; L. Shen, {\em Tensor invariants, saturation problems, and Dynkin automorphisms.}
   Adv. Math. 285 (2015), 629--657.
  \bibitem[H]{H} J.  Hong, {\em Conformal blocks, Verlinde formula and diagram automorphisms.} Adv. Math. 354 (2019), 106731, 50 pp.

  \bibitem[Ho]{Ho} R.  Howe, {\em
    Perspectives on invariant theory: Schur duality, multiplicity-free actions and beyond,}
    The Schur lectures (1992) (Tel Aviv), Israel Math. Conf. Proc. 8, Bar-Ilan Univ., Ramat Gan, 1995, pp. 1-182. 
   
   

 \bibitem[KLP]{KLP}  S. Kumar; G. Lusztig; D. Prasad,  {\em Characters of simplylaced nonconnected groups versus characters of
   nonsimplylaced connected groups.}  Representation theory, 99--101, Contemp. Math., 478, Amer. Math. Soc., Providence, RI, 2009.
   

\bibitem[Kap]{Kap}M. Kapovich: {\em Buildings and tensor product multiplicities}, appendix to the article of S. Kumar:
  {\em A survey of the additive eigenvalue problem},
Transformation Groups 19, 1051-1148 (2014)

\bibitem[KT]{KT} A. Knutson; T. Tao: {\em The honeycomb model of $\GL_n(\C)$
  tensor products. I. Proof of the saturation conjecture.}
  J. Amer. Math. Soc. 12 (1999), no. 4, 1055¡V1090.

  
\bibitem[Ku]{Ku} S. Kumar {\em Tensor Product Decomposition},
  Proceedings of the International Congress of Mathematicians, Hyderabad, India, 2010. 




\bibitem[Ku2]{Ku2} S. Kumar:
  {\em A survey of the additive eigenvalue problem (with an appendix by M. Kapovich)},
Transformation Groups 19, 1051-1148 (2014)

\bibitem[Ku-St]{Ku-St} S. Kumar and J. Stembridge; {\em Special isogenies and Tensor product multiplicities}, Int. Math. Res. Not., IMRN (2007), no. 20.


\bibitem[Pr]{Pr} D. Prasad, {\em Multiplicities under basechange: finite field case}, arXiv:1909.01850.

\bibitem[St]{St} J. Stembridge, {\em Multiplicity-free products and restrictions of Weyl characters.} Represent. Theory 7 (2003), 404-439.

\bibitem[LiE]{lie} {\sf LiE},  A  Computer  algebra  package  for  Lie  group
  computations, available at \\ \url{http://wwwmathlabo.univ-poitiers.fr/\~maavl/LiE/}


\end{thebibliography}
\end{document}